\documentclass[a4paper,10pt]{article}
\usepackage[plainpages=false]{hyperref}
\usepackage{amsfonts,latexsym,rawfonts,amsmath,amssymb,amsthm, mathrsfs, lscape, calligra}
\usepackage{verbatim}
\usepackage{enumerate}
\usepackage{centernot}
\usepackage{mathtools}
\usepackage{stmaryrd}
\usepackage{tikz}
\usepackage{authblk}
\usepackage{tikz-cd}
\usepackage{marginnote}
\usepackage{mathtools}

\usepackage[all]{xy}
\usepackage{graphicx,psfrag}

\usepackage{array, tabularx, color}

\usepackage{setspace}

\newtheorem{thm}{Theorem}[section]
\newtheorem{cor}[thm]{Corollary}
\newtheorem{lem}[thm]{Lemma}
\newtheorem{clm}[thm]{Claim}
\newtheorem{prop}[thm]{Proposition}

\newtheorem{conj}[thm]{Conjecture}

\theoremstyle{remark}
\newtheorem{rmk}[thm]{Remark}

\theoremstyle{definition}
\newtheorem{defi}[thm]{Definition}

\def \E {\mathcal{E}}
\def \F {\mathcal{F}}
\def \O {\mathcal{O}}

\def \bp {\bar{\partial}}
\def\e {\underline e}

\def \C {\mathbb C}
\def \Z {\mathbb Z}
\def \R {\mathbb R}

\def \F {\mathcal F}
\def \E {\mathcal E}

 \def \Q {\mathcal Q}
\def \P {\mathbb P}
\def \H {\mathcal H}

\def \bp {\bar{\partial}}

\def \R {\mathcal{R}}
\def \O {\mathcal{O}}
\def \l0 {\lim_{r\rightarrow 0}}

\def \Id {\text{Id}}

\def \Sing {\text{Sing}}

\def \G {\mathcal G}
\def \I {\mathcal I}

\DeclareMathOperator{\rk}{rank}

\def \C {\mathbb C}
\def \Z {\mathbb Z}
\def \R {\mathbb R}

\def \F {\mathcal F}

 \def \Q {\mathcal Q}
\def \P {\mathbb P}
\def \H {\mathcal H}

\def \bp {\bar{\partial}}

\def \R {\mathcal{R}}
\def \O {\mathcal{O}}
\def \l0 {\lim_{r\rightarrow 0}}

\def \E {\mathcal{E}}

\numberwithin{equation}{section}
\begin{document}

\title{Algebraic tangent cones of reflexive sheaves}
\author{Xuemiao Chen\thanks{xuemiao.chen@stonybrook.edu}, Song Sun\thanks{sosun@berkeley.edu. }}

\maketitle
\begin{abstract}
We study the notion of algebraic tangent cones at singularities of reflexive sheaves. These correspond to extensions of reflexive sheaves across a negative divisor. We show the existence of optimal extensions in a constructive manner, and we prove the uniqueness in a suitable sense. The results here are an algebro-geometric counterpart of our previous study on singularities of Hermitian-Yang-Mills connections.
\end{abstract}
\section{Introduction}
The goal of this paper is to study a complex-algebraic object that comes out of  our study of singularities of Hermitian-Yang-Mills connections (\cite{CS1, CS2}). The discussion here will be purely complex-algebraic, and the connection with the previous results will be given by Conjecture \ref{Conjecture-1.6}.  Let $B\subset \C^n$ be the unit ball, and let $\E$ be  an  reflexive  analytic coherent sheaf over of $B$. Let $\hat B$ be the blow-up of $B$ at $0$. We will use the following notation

\begin{itemize}
\item $\pi: B\setminus \{0\}\rightarrow \C\P^{n-1}$, and  $\pi': \C^n\setminus\{0\}\rightarrow\C\P^{n-1}$ are the natural projection maps;
\item $\psi: B\setminus\{0\}\rightarrow B$ and $\psi': \C^n\setminus\{0\}\rightarrow \C^n$ are the natural inclusion maps;
\item $p: \hat B\rightarrow B$ and $\phi: \hat B\rightarrow \C\P^{n-1}$ are the natural projection maps;
\item $D:=p^{-1}(0)$ is the exceptional divisor, and $\iota: D\rightarrow \hat B$ is the natural inclusion map. 
\end{itemize}

\begin{defi}
\begin{itemize}
\item An \emph{extension} of $\E$ at $0$ is a reflexive sheaf $\hat\E$ over $\hat B$  such that $\hat \E|_{\hat B\setminus D}$ is isomorphic to $(p^*\E)|_{\hat B \setminus D}$. Define $\mathcal A$ to be the set of  isomorphism classes of all extensions of $\E$ at $0$;
\item An \emph{algebraic tangent cone} of $\E$ at $0$ is a torsion-free sheaf  $\underline{\hat{\E}}$ over $D$ such that  $\underline{\hat{\E}}=\iota^*\hat\E$ for some $\hat \E \in \mathcal{A}$.
\end{itemize}
\end{defi}

To justify the terminology ``algebraic tangent cone", we notice that $\psi'_*\pi'^*\underline{\hat \E}$ defines a torsion-free sheaf on $\C^n$ with a natural $\C^*$ equivariant action. It would be more natural to call $\psi'_*\pi'^*\underline{\hat \E}$ the algebraic tangent cone, but we have chosen to call $\underline{\hat \E}$ itself an algebraic tangent cone just for convenience of our presentation.  In Conjecture \ref{Conjecture-1.6} below we shall also make connection with \emph{analytic tangent cones} of Hermitian-Yang-Mills connections. 

We remark that $\mathcal A$ is easily seen to be non-empty for one can simply take $(p^*\E)^{**}$ as an extension of $\E$ at $0$. Since the divisor line bundle $[D]$ is trivial on $\hat B\setminus D$, we know that if $\hat \E$ is an extension, then $\hat\E\otimes [D]^{\otimes k}$ is also an extension for all $k\in \Z$. In particular,  if $\underline{\hat \E}$ is an algebraic tangent cone, so is $\underline{\hat \E}\otimes \O(k)$ for all $k\in \Z$. It is easy to see that $\psi'_*\pi'^*(\underline{\hat \E}\otimes \O(k))$ is isomorphic to $\psi'_*\pi'^*\underline{\hat{\E}}$. 

\begin{defi}
Two extensions ${\hat \E}_1$ and ${\hat\E}_2$  of $\E$ at $0$ are \emph{equivalent} if ${\hat\E}_1$ is isomorphic to ${\hat \E}_2\otimes [D]^{\otimes k}$ for some $k\in \Z$.
\end{defi}

Since $D$ is of co-dimension one in $\hat B$, in general $\mathcal A$ consists of more than one element. For example, as we shall show in Proposition \ref{Hecktransform} given any extension $\hat\E$, then a saturated subsheaf of $\underline{\hat\E}$ determines a \emph{Hecke transform} of $\hat\E$ (see Definition \ref{Definition2.3}), which in general may be different from $\hat\E$. Our goal is to define and find an \emph{optimal} extension in the following sense.

Given any $\hat \E\in\mathcal A$, we let $0\subset \underline\E_1\subset \cdots \underline\E_m \subset \underline{\hat\E}$ be the Harder-Narasimhan filtration of $\underline{\hat\E}$ (with respect to the obvious polarization $\O(1)\rightarrow \C\P^{n-1})$. Denote by $\mu_k$ the slope of $\underline \E_k/\underline \E_{k-1}$ which is strictly decreasing in $k$. We define a function
$\Phi: \mathcal A\rightarrow \mathbb Q_{\geq0}$ by setting  
$$\Phi(\hat\E)=\mu_1-\mu_m. $$
Then $\underline{\hat \E}$ is semistable if and only if $\Phi(\hat\E)=0$. In general $\Phi(\hat\E)$ measures the deviation of  the algebraic tangent cone $\underline{\hat\E}$ from being semistable. The naive goal is to find an extension so that the algebraic tangent cone is semistable. However, by Theorem \ref{main} below, it is easy to see this can not be achieved in general.   Instead we make the following definition

\begin{defi}
We say an extension $\hat\E$ is
\begin{itemize}
\item \emph{optimal} if $\Phi(\hat\E)\in [0, 1)$;
\item \emph{semistable} if $\Phi(\hat\E)=0$.
\end{itemize}
\end{defi}

For any torsion-free sheaf $\underline\F$ on $\C\P^{n-1}$ we denote by $Gr(\underline\F)$ the graded torsion-free sheaf associated to the Harder-Narasimhan filtration of $\underline\F$ and by $Gr^{HNS}(\underline \F)$ the graded torsion-free sheaf associated to a Harder-Narasimhan-Seshadri filtration of $\underline\F$. The main result we shall prove is 
\begin{thm}\label{main}
Given a reflexive coherent  sheaf $\E$ on $B$, the following holds
\begin{enumerate}[(I).]
\item An optimal extension always exists. More precisely, given any $\hat \E\in \mathcal A$, there are finitely many Hecke transforms that transform $\hat\E$ into an optimal one;
\item Suppose $\hat \E_1\in\mathcal{A}$ and $\hat\E_2\in \mathcal A$ satisfy that $\Phi(\hat\E_1)+\Phi(\hat\E_2)<1$, then $\hat\E_1$ and $ \hat\E_2$ are equivalent. In particular, if there is one semistable extension, then it is the unique optimal extension up to equivalence;
\item Suppose $\hat \E_1$ and $\hat\E_2\in\mathcal A$ are both optimal extensions,  then there is a $k\in \Z$ such that $\hat \E_1$ and $\hat \E_2\otimes [D]^{\otimes k}$ differ by a Hecke transform of special type (see Definition \ref{defi3.4}). In particular,  
$$\psi'_*\pi'^*(Gr(\underline{\hat\E_1}))\simeq \psi'_*\pi'^*(Gr(\underline{\hat \E_2}));$$
\item Suppose $\E$ is homogeneous, i.e. $\E\simeq \psi_*\pi^*\underline \E$ for some reflexive sheaf $\underline \E$ on $\C\P^{n-1}$, then there exists an optimal extension $\hat{\E}\in \mathcal{A}$ with
$$
\underline{\hat{\E}}\cong\widetilde{Gr}(\underline \E),$$
where $\widetilde{Gr}(\underline \E)$ denotes the graded sheaf determined by the partial Harder-Narasimhan filtration of $\underline \E$ (see Section \ref{proofIV}). In particular, $$\psi'_*\pi'^*(Gr^{HNS}(\underline{\hat\E}))\simeq \psi'_*\pi'^*(Gr^{HNS}(\underline \E)).$$
\end{enumerate}
\end{thm}

\begin{rmk}
It is very crucial that the normal bundle of $D$ is negative in our case but it is not crucial that $D$ is $\C\P^{n-1}$.
\end{rmk}

\begin{rmk}
The above result also yields some tree-like structure on $\mathcal A$, which does not seem obvious to see without using the notion of optimal extensions. Also notice $\mathcal A$ itself may contain continuous moduli. For example, in the case when $n=2$ and $\E$ is the trivial rank 2 sheaf on $\C^2$, any extension $\hat\E$ restricts to $\O(k_1)\oplus \O(-k_2)$ on $\C\P^1$ for some $k_1,k_2\in\mathbb{Z}_{\geq 0}$. By \cite{Gasparim}, under the restriction $k_1=k_2>1$, there is a generically $2k_1-3$ dimensional moduli of isomorphism classes of extensions. 
\end{rmk}

We give here a simple example illustrating the above statements, and we refer to Section \ref{Examples} for more examples.  Let $\underline\F$ be the locally free sheaf given by $\O\oplus \mathcal T\C\P^{n}$ for $n\geq 2$, and let $\E=\psi_*\pi^*\underline\F$. Then $\hat \E_1:=\phi^*\underline \E$ is an extension of $\E$ and the corresponding algebraic tangent cone is  $\underline{\hat\E}_1=\underline \F$. Since $\Phi(\hat \E_1)=\frac{n+1}{n}$, we know $\hat\E_1$ is not optimal. Applying Hecke transform to $\hat \E_1$ along the subsheaf $\mathcal T\C\P^{n}$ (which is fairly trivial in this case), we get a new extension $\hat \E_2$ with $\underline{\hat\E}_2=\mathcal T\C\P^{n}\oplus \O(1)$. Since $\Phi(\hat \E_2)=\frac{1}{n}\in [0, 1)$, $\hat \E_2$ is also an optimal extension. Moreover, by (II) above, there is \emph{no} semistable extension of $\E$.  Also the strict uniqueness of optimal extensions up to equivalence is not true in this example, since one can easily find another optimal extension $\hat\E_3$ with $\underline{\hat \E}_3=\mathcal T\C\P^n\oplus \O(2)$, and $\Phi(\hat\E_3)=\frac{n-1}{n}\in [0, 1)$. This shows that (II) is sharp. However, it is clear that 
$$\psi'_*\pi'^*(Gr(\underline{\hat\E}_2))\simeq \psi'_*\pi'^*(Gr(\underline{\hat\E}_3))\simeq \E$$
which is compatible with (IV) above. 

One of the reasons that we consider the associated graded sheaf in the above result is to connect with our previous results \cite{CS1, CS2}. It is reasonable to make the following conjecture. For related concepts  we refer the interested readers to \cite{CS2}. 

\begin{conj}\label{Conjecture-1.6}
Let $A$ be an admissible Hermitian-Yang-Mills connection on $(\E, B)$ and $\hat{\E}$ be any chosen optimal extension of $\E$ at $0$. Then there is a unique analytic tangent cone $(\E_\infty, A_\infty, \Sigma^{an}, \mu)$ on $\C^n$ of $(\E,A)$ at $0$, where $(\E_\infty, A_\infty)$ is a Hermitian-Yang-Mills cone, $\Sigma^{an}$ is the bubbling set, $\mu$ is the limiting measure, and moreover 
\begin{itemize}
\item $\E_\infty\simeq\psi'_*\pi'^*((Gr^{HNS}(\underline {\hat \E}))^{**})$;
\item $\Sigma^{an}=\pi'^{-1}(\Sing(Gr^{HNS}(\underline{\hat{\E}})))$;
\item for each irreducible component of $\Sigma^{an}$ of pure complex codimension $2$, the analytic multiplicity assigned by $\mu$ and the algebraic multiplicity are equal.
\end{itemize} 
\end{conj}
\begin{rmk}
By Theorem \ref{main} (III), $Gr^{HNS}(\underline{\hat{\E}})$ is independent of the choice of an optimal extension up to tensoring with $\O(k)$. Namely, we already have uniqueness in the algebraic-geometric side. 
\end{rmk}

Combining Theorem \ref{main} and the main results in \cite{CS1, CS2}, we have proved the above conjecture in the case when $\E$ is homogeneous with an isolated singularity at $0$, i.e. when $\E\simeq \psi_*\pi^*\underline \E$ for some locally free sheaf $\underline \E$ on $\C\P^{n-1}$. When $\E$ admits an algebraic tangent cone $\hat\E$ which is a stable vector bundle(and then it must be the unique optimal extension up to equivalence by Theorem \ref{main}), we also know that $\E_\infty\simeq \psi'_*\pi'^*\underline{\hat{\E}}$, see Theorem $1.3$ in \cite{CS1}.  Conjecture \ref{Conjecture-1.6} improves and generalizes the conjectures in \cite{CS1, CS2}. Notice  in \cite{CS1, CS2} we restricted to the case of isolated singularities due to technical reasons;  for general admissible Hermitian-Yang-Mills connections, the existence of analytic tangent cones is still true by \cite{Tian}.

\

\textbf{Acknowledgments.} We are grateful to  Richard Thomas for pointing out that our original differential geometric construction in Section 2.2 is exactly given by what we now call the Hecke transform, which considerably simplifies  the language in our proof. We also thank Simon Donaldson and Thomas Walpuski for their interest in this work. 

Both authors are supported by  the Simons Collaboration Grant on Special Holonomy in Geometry, Analysis, and Physics (488633, S.S.).   S. S. is partially supported by an Alfred P. Sloan fellowship and NSF grant DMS-1708420. 

\section{Hecke transform of reflexives sheaves}

\subsection{The case of sub-bundles}\label{Section2.1}
Let $M$ be a complex manifold and $D$ be a smooth hypersurface in $M$. Let $E$ be a holomorphic vector bundle on $M$ and denote $\underline E:=E|_D$. Let $\underline F$ be a sub-bundle of $\underline E$. Let $\underline Q$ denote the quotient bundle $\underline E/\underline F$ and $\underline p: \underline E\rightarrow \underline Q$ the natural projection map. Then we have the following short exact sequence of vector bundles on $D$
\begin{equation} \label{eqn1-2}
0\rightarrow \underline F\rightarrow \underline E\xrightarrow{\underline p} \underline Q\rightarrow 0.
\end{equation}
We will describe below a construction, called the \emph{Hecke transform} along $\underline F$, that yields another vector bundle $E'$ on $M$, which is isomorphic to $E$ on $M\setminus D$, such that the restriction  $\underline E':=E'|_D$ fits into an extension of the form 
\begin{equation}\label{eqn1-3}
0\rightarrow \underline Q\otimes N_D\rightarrow \underline E'\rightarrow \underline F\rightarrow 0, 
\end{equation}
where $N_D$ is the normal bundle of $D$ in $M$. In the next section we shall re-interpret it in terms of more complex-analytic language, which makes the construction more natural and generalizes to the case of coherent sheaves. 

To start the construction, we choose an open cover $\{U_\alpha\}$ of a neighborhood $U$ of $D$, such that $E|_{U_\alpha}$ admits a trivialization given by holomorphic sections $e_{\alpha, 1}, \cdots, e_{\alpha, r}$, and such that if we denote $\e_\alpha^j:=e_\alpha|_{V_\alpha}$ where $V_\alpha:=U_\alpha\cap D$, then $\e_{\alpha, 1}, \cdots, \e_{\alpha, s}$ give a holomorphic trivialization of $F|_{V_\alpha}$, and $\underline p(e_{\alpha, s+1}), \cdots, \underline p(e_{\alpha, r})$ give a holomorphic trivialization of $\underline Q|_{V_\alpha}$.  We may also assume that the divisor line bundle $[D]$ has a local trivialization $t_\alpha$ on each $U_\alpha$. Choose a defining section $s$ of $[D]$ so that we can write $s=s_\alpha t_\alpha$ over each $U_\alpha$ with $s_\alpha$ vanishing on $D$ with exactly order one. 

On the intersection $U_{\alpha\beta}:=U_\alpha\cap U_\beta$, we can write the transition function of $E$ as 
$$
\Phi_{\alpha\beta}=\begin{bmatrix}
f_{\alpha\beta} & g_{\alpha\beta}   \\
h_{\alpha\beta}      &q_{\alpha\beta}.
\end{bmatrix}
$$
Denote $V_{\alpha\beta}:=U_{\alpha\beta}\cap D$. Then the fact that $\underline F$ is a sub-bundle of $\underline E$ implies that $h_{\alpha\beta}|_{V_{\alpha\beta}}=0$, and $g_{\alpha\beta}|_{V_{\alpha\beta}}$ defines the extension class in $\text{Ext}^1(\underline Q , \underline F)$ corresponding to the short exact sequence (\ref{eqn1-2}). 

Now define a new holomorphic basis of $E|_{U_\alpha\setminus D}$ by setting $e_{\alpha, j}'=e_{\alpha, j}$ for $j\leq s$ and $e_{\alpha, j}'=s_\alpha e_{\alpha, j}$ for $j\geq s+1$. Then with respect to the new basis, the new transition matrix becomes
$$
\Phi_{\alpha\beta}'=\begin{bmatrix}
f_{\alpha\beta} & g_{\alpha\beta}s_\alpha  \\
h_{\alpha\beta}s_\beta^{-1}      &q_{\alpha\beta} s_{\alpha}s_{\beta}^{-1}
\end{bmatrix}.
$$

Now the entries of this matrix extend to be well-defined holomorphic functions across $V_{\alpha\beta}$. Hence it defines a holomorphic vector bundle on $M$, which is our desired $E'$.  Moreover, since $s_{\alpha}s_{\beta}^{-1}$ is the transition function of the line bundle $[D]$, by adjunction formula, we see that by restricting to $D$, the right bottom component of $\Phi'_{\alpha\beta}$ gives the transition matrix for $\underline Q\otimes N_D^{-1}$. It is also clear that the whole matrix restricting to $D$ is now a lower triangular matrix, so it is obvious that the exact sequence (\ref{eqn1-3}) holds. 

One can check by definition that there is a well-defined vector bundle isomorphism from $E'$ to $E$ on $M\setminus D$, since by construction locally a holomorphic section of $E'$ is a holomorphic section of $E$ such that when restricting to $D$ it belongs to $\underline F$. One can also check that the isomorphism class of $E'$ does not depend on the choices made. It is also clear from the construction in the next subsection. 

\begin{rmk}
When $\dim M=1$, $D=\{x\}$, $\underline F$ is a subspace of $E|_x$. In this case the above construction is usually referred to as  the ``\emph{elementary modification}" or ``\emph{Hecke modification}" in the literature, and this justifies our choice of terminology.  
\end{rmk}

\subsection{General case}
Now we  move on to the general case of coherent sheaves, using a more complex-algebraic language (which is kindly pointed out to us by Richard Thomas). We again suppose $M$ is a smooth complex manifold and $D$ is a smooth hypersurface. Let $\iota: D\rightarrow M$ be the natural inclusion map, and $\E$ be a reflexive sheaf on $M$.  By  Lemma $3.24$ in \cite{CS1}, we know that $\underline \E:=\iota^*\E$ is a torsion-free coherent sheaf on $D$. 

 Let $\underline\F$ be a subsheaf of $\underline \E$ and $\underline\Q$ be the quotient sheaf. Denote $p: \mathcal E \rightarrow \iota_*(\underline\Q)$ to be the map given by the composition of  the natural surjective map $\E\rightarrow \iota_*\underline\E$ with the natural map $\iota_* \underline \E \rightarrow \iota_* \underline Q$.

\begin{lem}\label{lem1.1}
$p$ is a surjective sheaf homomorphism.
\end{lem}
\begin{proof}
It suffices to show the map $\iota_*\underline \E \rightarrow \iota_* \underline \Q$ is surjective. By definition we have the following exact sequence 
$$
0\rightarrow  \underline\F \rightarrow \underline\E \rightarrow \underline\Q \rightarrow 0. 
$$
Since $\iota: D \hookrightarrow M$ is obviously \emph{Stein}, namely, the pre-image of a Stein open set is Stein, the higher direct image $\R^i(\iota_* \underline \F)$ vanishes for $i\geq 1$. In particular, the following is exact
$$
0\rightarrow \iota_*\underline\F \rightarrow \iota_*\underline \E \rightarrow \iota_*(\underline\Q) \rightarrow 0.
$$
\end{proof}

\begin{defi}\label{Definition2.3}
We define the \emph{Hecke transform} $\E'$ of $\E$ along $\underline \F$ to be the kernel of the map $p$. 
\end{defi}

By definition, $\E'$ lies in the following short exact sequence
\begin{equation} \label{eqn1-5}
0\rightarrow \E'\rightarrow \E\rightarrow \iota_*\underline\Q\rightarrow 0. 
\end{equation}
In particular $\E'$ is a subsheaf of $\E$ which is  isomorphic to $\E$ over $M\setminus D$. In particular, it must be torsion-free. It is easy to check by definition that when $\E$ is locally free over $M$ and $\underline\Q$ is locally free over $D$, this agrees with the construction in the previous subsection. 

\begin{lem}
$\E'$ is reflexive if $\underline\F$ is saturated in $\underline \E$ or equivalently $\underline \Q$ is torsion-free. 
\end{lem}
\begin{proof}
By Equation (\ref{eqn1-5}), we have the following exact sequence 
$$0\rightarrow (\E')^{**}/\E' \rightarrow \iota_* \underline{\Q}.$$
Since $\mathcal{I}_D \cdot \iota_* \underline \Q=0$, we have $\mathcal{I}_D\cdot ((\E')^{**}/\E')=0$. Then we have 
$$
(\E')^{**}/\E'=\iota_*\iota^*((\E')^{**}/\E')$$
and the following exact sequence
$$
0\rightarrow \iota^* ((\E')^{**}/\E' ) \rightarrow \iota^*\iota_*\underline \Q=\underline{\Q}.
$$
Since $\E'$ is torsion-free and locally free outside $D$, $\text{Supp}((\E')^{**}/\E')$ has codimension $1$ in $D$, which implies $\iota^* ((\E')^{**}/\E' )$ is a torsion sheaf.  Since $\underline\Q$ is torsion-free, by the exact sequence above, we have
$\iota^* ((\E')^{**}/\E' )=0$ which implies $(\E')^{**}/\E'=0$. This finishes the proof. 
\end{proof}

In our later applications we will always assume $\underline\F$ is saturated in $\underline \E$. The following proposition is a generalization of (\ref{eqn1-3}).

\begin{lem}\label{lem1.2}
There exists the following exact sequence
\begin{equation}\label{eqn1-1}
0\rightarrow \I_D\cdot\E\rightarrow \E' \rightarrow \iota_* \underline\F\rightarrow 0.
\end{equation}
\end{lem}
\begin{proof}
By definition $\E'$ is exactly the pre-image of $\iota_*\underline\F$ under the natural map $\E\rightarrow \iota_*\underline \E$.  So we have a natural surjective map $\E'\rightarrow \iota_*\underline \F$. The kernel of this map agrees with the kernel of the map $\E\rightarrow \iota_*\underline \E$, which is exactly $\I_D \cdot \E$. This finishes the proof. 
\end{proof}

Denote $\underline \E'=\iota^*\E'$.
\begin{prop}\label{Hecktransform}
There exists the following exact sequence 
$$
0\rightarrow \underline\Q\otimes \mathcal N_D^{*}\rightarrow \underline \E' \rightarrow \underline \F \rightarrow 0, 
$$
where $\mathcal N_D^*\simeq \I_D/\I_D^2$ is the locally free sheaf associated to the co-normal bundle of $D$. 
\end{prop}
\begin{proof}
Applying $\iota^*$ to (\ref{eqn1-1}) we get the exact sequence 
\begin{equation}\label{eqn1-4}
 \iota^*(\I_D\cdot \E) \xrightarrow{\psi}\underline \E' \rightarrow\iota^* \iota_*\underline\F=\underline \F\rightarrow 0.
\end{equation}
It suffices to prove $\text{Ker}(\psi)=\underline Q\otimes \mathcal N_D^*$.  By definition, $\psi$ comes from the map $\I_D \cdot \E\rightarrow \E'$ by tensoring with $\O_D$, so the kernel is given by  $\I_D\cdot \E' /\I_D^2 \cdot \E $. Since $\I_D$ is locally free, we have the following exact sequence
$$
0\rightarrow \I^2_D \cdot \E  \rightarrow \I_D\cdot \E' \rightarrow \I_D \otimes \iota_*\underline\F\rightarrow 0.
$$
This implies that as $\O_M$-modules, we have 
$$\I_D\cdot \E' /\I_D^2 \cdot \E=\I_D \otimes \iota_*\underline\F= \iota_*(\underline \F \otimes \mathcal N_D^*) $$
It is direct to check that the inclusion of $\text{Ker}(\psi)$ in $\iota^*(\I_D\cdot \E)$ is given by the natural map 
$$\iota_*\underline\F\otimes \mathcal{N}_D^*\rightarrow \underline \E\otimes \mathcal{N}_D^*$$
under the natural identification $\iota^*(\mathcal{I}_D\cdot \E)=\underline\E\otimes \mathcal{N}_D^* $. Hence we see the image of $\psi$ is given by 
$$ (\underline \E\otimes \mathcal N_D^*)/ (\underline\F\otimes \mathcal N_D^*)=\underline\Q\otimes \mathcal N_D^*.$$
\end{proof}

Now we will discuss some interesting properties of the Hecke transform. Let $\E''$ be the Hecke transform of $\E'$ along $\underline \Q\otimes \mathcal N_D^*$. 

\begin{lem} \label{lem2.7}
The Hecke transform is an involution up to twisting by $[D]$ in the sense that $\E''\cong \E (-[D])$.
\end{lem}
\begin{proof}
By definition and Proposition \ref{Hecktransform}, $\E''$ fits into the following exact sequence
$$
0\rightarrow \E''\rightarrow \E'\rightarrow \iota_*(\underline \E'/(\Q \otimes \mathcal N_D^*))=\iota_*\underline\F\rightarrow 0, 
$$
and the map $\E'\rightarrow\iota_*\underline \F$ agrees with the map in (\ref{eqn1-1}).
By Lemma \ref{lem1.2}, $\E''$ is isomorphic to $\I_D\cdot\E$. 
\end{proof}

More generally, we can take a subsheaf of $\underline\Q\otimes \mathcal N_D^{*}$ which has the form $(\underline\E_1/\underline \F) \otimes \mathcal N_D^*,$
where $\underline \E_1\subset \iota^*\E$ is a saturated subsheaf with $\underline \F \subset \underline \E_1$. Let  $\E''_1$ be the Hecke transform of $\E'$ along $(\underline\E_1/\underline \F) \otimes \mathcal N_D^*$ and $\E'_1$ be the  Hecke transform of $\E$ along $\underline \E_1$. Then the following involution property holds.  

\begin{prop}
$\E''_1\simeq \I_D \cdot\E_1'$.
\end{prop} 

\begin{proof}
We have the following commutative diagram
\[\begin{tikzcd}
0\arrow{r}& \E_1''\arrow{r}\arrow{d}&\E'\arrow{r}\arrow{d}{=}&\iota_*(\underline\E'/((\underline\E_1/\underline \F)\otimes \mathcal N_D^*))\arrow{r} \arrow{d}&0\\
0\arrow{r}& \E''= \I_D\cdot \E \arrow{r}
&\E' \arrow{r}
&\iota_*(\underline \E' /(\underline\Q\otimes \mathcal N_D^{*}) \arrow{r} &0
\end{tikzcd}
\]
where the first row  is by definition and the second row is by  Lemma \ref{lem2.7}. This implies the following exact sequence 
$$
0\rightarrow (\I_D\cdot \E )/\E_1'' \rightarrow \E'/\E_{1}'' =\iota_*(\underline\E'/((\underline\E_1/\underline \F)\otimes \mathcal N_D^*)) \rightarrow  \iota_*(\underline \E' /(\underline\Q\otimes \mathcal N_D^{*})\rightarrow 0.
$$
As a result, we have 
$$(\I_D\cdot \E )/\E_1''=\iota_*(\underline\Q\otimes \mathcal N_D^*)/\iota_*(\underline \E_1/\underline \F \otimes \mathcal N_D^*)=\iota_*((\underline \E/\underline\E_1)\otimes \mathcal N^*_D)
$$
which implies the following exact sequence 
$$
0\rightarrow \E_1''\rightarrow \I_D\cdot \E  \rightarrow \iota_*(\underline \E/ \underline \E_1 \otimes \mathcal N^*_D)\rightarrow 0.
$$
By definition, we also have 
$$
0\rightarrow \E_1 \rightarrow \E \rightarrow \iota_*(\underline \E/\underline \E_1)\rightarrow 0.
$$
Since $\I_D$ is locally free, we have 
$$
0\rightarrow\I_D\cdot \E_1  \rightarrow \I_D \cdot \E  \rightarrow \iota_*(\underline \E/\underline \E_1)\otimes \I_D=\iota_*(\underline \E /\underline \E_1 \otimes \mathcal N_D^*) \rightarrow 0.
$$
This finishes the proof. 
\end{proof}

\section{Proof of the Main Theorem}
\subsection{Proof of (I)}
 We begin with a simple observation.
\begin{lem}\label{lem1.8}
The image of the map $\Phi: \mathcal A\rightarrow \mathbb Q_{\geq 0}$ is discrete. In particular, a minimizer of $\Phi$ always exists. 
\end{lem}

\begin{proof}
By definition, 
$$\mu_i = \frac{\int_{D}c_1(\E_i/\E_{i-1})\cup c_1(\O(1))^{n-2}}{\rk(\E_i/\E_{i-1})}\in (\rk(\E)!)^{-1}\mathbb{Z}.$$
This  implies for any extension $\hat \E$,
$
\Phi(\hat{\E}) \in (\rk(\E)!)^{-1}\mathbb{Z}_{\geq 0}.
$
\end{proof}

Now let $\hat\E\in \mathcal A$. Let $0\subset \underline\E_1\subset \cdots \underline\E_m = \underline{\hat \E}$ be the Harder-Narasimhan filtration of $\underline{\hat\E}$. In the following, for each $k<m$ we always denote by $\hat{\E}^k$ to be the Hecke transform of $\hat \E$ along $\underline \E_k$ and denote $\underline{\hat{\E}^k}=\iota^*\hat{\E}^k$.  Given any sheaf $\underline \F$ over $\C\P^{n-1}$, we also denote  
$$\underline\F(j):=\underline\F\otimes \mathcal O(j).$$

\begin{lem}\label{ErrorComparison}
$\Phi(\hat{\E}^k)\leq \max\{\mu_{k+1}-\mu_m, \Phi(\hat{\E})-1, \mu_{k+1}-\mu_k +1, \mu_1-\mu_k\}
$
for any $k$. 
\end{lem}

\begin{proof}
By Corollary \ref{Hecktransform}, we have the following exact sequence 
\begin{equation}\label{eqn1.1}
0\rightarrow (\underline {\hat\E}/\underline \E_{k})(1)\rightarrow \underline{\hat{\E}^k}\rightarrow \underline\E_{k} \rightarrow 0.
\end{equation}
Let $0\subset \underline \E_1'\subset\cdots \underline\E'_{m'}=\underline{\hat \E^k}$ be the Harder-Narasimhan filtration of $\underline{\hat\E^k}$. Denote the slope of $\underline \E_i'/\underline \E'_{i-1}$ by $\mu_i'$.  By Equation (\ref{eqn1.1}), $\underline \E_1'$ fits into the following exact sequence 
$$
0\rightarrow \underline \G_1 \rightarrow \underline \E_1' \rightarrow \underline \G_2\rightarrow 0.
$$
where $\underline \G_1$ is a subsheaf $(\underline{\hat \E} /\underline\E_{k})(1)$ and $\underline \G_2$ is a subsheaf of $\underline\E_k$. Since $\underline \E_{k+1} /\underline \E_{k}$ is the maximal destabilizing subsheaf of $\underline{\hat{\E}}/\underline \E_{k}$, we have
$$\mu(\underline \G_1) \leq \mu_{k+1}+1$$
  Similarly 
 $$\mu(\underline \G_2) \leq \mu_1.$$ Then one has 
\begin{equation}\label{1}
\mu_1' \leq \max\{\mu_{k+1}+1, \mu_1\}.
\end{equation}
By taking the dual of Equation (\ref{eqn1.1}), one has the following exact sequence
$$
0\rightarrow \underline { \E}_{k}^{*} \rightarrow (\underline{ \hat{\E}}^k) ^* \rightarrow (\underline{\hat \E}/ \underline\E_{k})^*(-1).
$$
Similarly $(\underline\E'_{m'}/\underline\E'_{m'-1})^*$ fits into the following exact sequence
$$
0\rightarrow \underline \H_1 \rightarrow (\underline\E'_{m'}/\underline\E'_{m'-1})^* \rightarrow \underline \H_2\rightarrow 0
$$
where $\underline \H_1$ is a subsheaf of $\underline \E_{k}^{*}$ and $\underline \H_2$ is a subsheaf of $ (\underline {\hat\E}/ \underline\E_{k})^*(-1)$. Similar to the above, we have
$$\mu(\underline \H_1) \leq -\mu_k$$
and
$$\mu(\underline \H_2) \leq -\mu_m-1.$$ 
Then one has 
\begin{equation}\label{2}
-\mu_{m'}' \leq \max\{-\mu_k, -\mu_m-1\}
\end{equation}
Combining Equation (\ref{1}) and (\ref{2}), we get
$$
\mu_1'-\mu'_{m'} \leq \max\{\mu_{k+1}-\mu_m, \mu_1-\mu_m-1, \mu_{k+1}-\mu_k +1, \mu_1-\mu_k\}
$$
This finishes the proof. 
\end{proof}

Now we prove Theorem \ref{main} (I). Since $\mathcal A$ is nonempty, we can fix an element $\hat\E\in \mathcal A$. If $\Phi(\hat\E)\geq1$, we apply Lemma \ref{ErrorComparison} to $\hat{\E}$ with $k=1$ and get    $$\Phi(\hat{\E}^1)\leq \max\{\mu_2-\mu_m, \Phi(\hat\E)-1, \mu_2-\mu_1+1\}\leq \Phi(\hat\E)-1. $$
If $\Phi(\hat{\E}^1)\geq 1$, we repeat the same process for $\hat{\E}^1$. After finitely many steps, we can get $\hat{\E}'\in\mathcal{A}$ with $0\leq \Phi(\hat{\E}')<1$. The following is also clear from Lemma \ref{ErrorComparison}.
\begin{cor}\label{lemma1.14}
Suppose $\hat{\E}\in \mathcal A$ is optimal, then $\hat{\E}^k$ is also optimal for all $k$. 
\end{cor}

\begin{defi}\label{defi3.4}
We say two optimal extensions $\hat{\E}$ and $\hat{\E}'$ differ by a Hecke transform of  special type if $\hat{\E}'$ is isomorphic $\hat{\E}^k$ for some $k$.
\end{defi}

\subsection{Proof of  (II) and (III)}
\subsubsection{Meromorphic extension of sections}
The goal in this subsection is to prove the following proposition that will be needed in our discussion later.  Let $s_D\in H^0(\hat B, [D])$ be a defining section of $D$ and let $\hat{\E}$ be any reflexive sheaf over $\hat{B}$. 
\begin{prop}\label{prop3-7}
Given any $s\in H^0(\hat{B}\setminus D, \hat{\E})$, there exists a $k$ such that  $s\otimes s_D^k$ extends to a holomorphic section  of $\hat\E(k[D])$ over  $\hat{B}$. In other words, $s$ is a meromorphic section of $\hat\E$. 
\end{prop}
\begin{rmk}
It is a key assumption here that $[D]$ is an exceptional divisor, since otherwise the statement is false. For example, if we consider $D=\{0\}\subset \Delta$, where $\Delta=\{|z|<1\}\subset \C$, and consider the trivial sheaf $\O$, then we have holomorphic functions on $\Delta\setminus\{0\}$ with an essential singularity at $0$ which can not extend to be meromorphic functions on $\Delta$.
\end{rmk}

\begin{proof}[Proof of the case $n=2$]
In this case $D=\C\P^1$, and $\hat \E$ is a locally free. Denote $\hat B_{t}:=p^{-1}(B_{t})$ where $B_{t}$ denote the ball of radius $t\in (0,1)$ centered at $0$.

It suffices to construct the following exact sequence over $\hat{B}_{\frac{1}{2}}$ for $k\in \mathbb{Z}$ large enough
$$
0\rightarrow R\rightarrow \O^{n_1} \rightarrow \hat{\E}^*(-k[D])\rightarrow 0.
$$
Indeed, given this exact sequence, by taking the double dual, we have 
$$
0\rightarrow  \hat{\E}(k[D])\rightarrow \O^{n_1} \rightarrow R^* \rightarrow  0.
$$
Then $s\otimes s_D^k|_{\hat{B}_{\frac{1}{2}}}\in H^{0}(\hat{B}_{\frac{1}{2}}\setminus D, \hat \E (k[D]))$ can be viewed as a section in $H^{0}(\hat{B}_{\frac{1}{2}}\setminus D, \O^{n_1})$. By Hartog's theorem for holomorphic functions, we know $H^{0}(\hat{B}_{\frac{1}{2}}\setminus D, \O^{n_1})=H^{0}(\hat{B}_{\frac{1}{2}}, \O^{n_1})$. Then $s\otimes s_D^k|_{\hat{B}_{\frac{1}{2}}} \in H^{0}(\hat{B}_{\frac{1}{2}}, \O^{n_1})$. By continuity, we have $s\otimes s_D^k|_{\hat{B}_{\frac{1}{2}}} \in H^{0}(\hat{B}_{\frac{1}{2}}, \hat \E(k[D]))$. 

Now we fix a K\"ahler metric $\hat \omega$ on $\hat{B}$. In order to construct the exact sequence above, it is equivalent to constructing a set of global generators for $\hat \E^*(-k[D])$ over $\hat B_{\frac{1}{2}}$ for $k$ large. This can done by the standard H\"ormander technique, see for example Theorem $5.1$ in \cite{Demailly}. Indeed, we know $\hat{B}_{\frac{1}{2}}$ is weakly pseudo-convex, and since $[D]|_D=\O(-1)$ is negative, one can easily construct a hermitian metric $h$ on $\hat\E^*(-k[D])$ for $k$ large,  such that 
$$
\sqrt{-1}F_{h_k}\geq C k \hat\omega\otimes\text{Id}.
$$ 
Now the conclusion follows from standard $L^2$ solution to the $\bp$-problem, using singular weight. 
\end{proof}

\begin{proof}[Proof of the general case] Suppose $n\geq 3$ and $\hat{\E}$ is a reflexive sheaf defined  $\hat{B}$. Let $S=\phi(\text{Sing}(\hat{\E}))\cap \overline{\hat B_{\frac{3}{4}}}$ and $\hat{S}=\phi^{-1}(S)\cap \hat{B}_{\frac{3}{4}}$. By replacing $\hat{B}_{\frac{3}{4}}$ with $\hat{B}$ which does not affect the argument, we can assume $S$ is a closed subset in $\C\P^{n-1}$ of Hausdorff of codimension at least $4$ and so is $\hat{S}$ in $\hat{B}$. Furthermore, $\text{Sing}(\hat\E)\subset \hat{S}$.

By Proposition $4$ in \cite{Siu}, it suffices to prove that for any $z\in \C\P^{n-1}\setminus S$, $s|_{\phi^{-1}(z)}$ is a meromorphic section of $\hat{\E}|_{ \phi^{-1}(z)}$. Indeed, given this, by Proposition $4$ in \cite{Siu}, we know $s$ is a meromorphic section of $\hat{\E}|_{\hat B \setminus \hat S}$ which is holomorphic outside $D$. Then for some $k$, $s\otimes s_D^k$ is a holomorphic section of $\hat{\E}(k[D])|_{\hat{B}\setminus \hat S}$. Since $\hat{S}$ has Hausdorff codimension at least $4$,  $s\otimes s_D^k$ further extends to be a section in $H^{0}(\hat{B}, \hat{\E}(k[D]))$ (see Lemma 3 in \cite{Shiffman}). Now we show $s|_{\phi^{-1}(z)}$ is a meromorphic section of $\hat{\E}|_{ \phi^{-1}(z)}$ for any $z\in \C\P^{n-1}\setminus S$. Since $S$ has Hausdorff of codimension at least $4$ in $\C\P^{n-1}$, we can choose a complex line $\C\P^{1}\subset \C\P^{n-1}$ which does not intersect $S$ but contains $z$. Let $\hat{B}^2=\phi^{-1}(\C\P^{1})$. Then $\hat{\E}|_{\hat{B}^2}$ is locally free. By the case $n=2$ proved above,  $s|_{\hat B^2\setminus (D\cap \hat B^2)}$ is a meromorphic section of $\hat{\E}$ over $\hat{B}^2$. In particular,  $s|_{\phi^{-1}(z)}$ is a meromorphic section of $\hat{\E}|_{ \phi^{-1}(z)}$. This finishes the proof. 
\end{proof}

\subsubsection{Uniqueness}\label{uniqueness}
We will prove (II) and (III) in this section.  Suppose $\hat\E$ and $\hat\E'$ are two optimal extensions of $\E$ at $0$. We denote $\underline{\hat \E}=\iota^* \hat{\E}$ and $\underline{\hat \E'}=\iota^* \hat{\E}'$. Let 
$$0\subset \underline \E_1\subset \cdots \underline \E_{m}=\underline{\hat \E}$$
and 
$$0\subset \underline \E'_1 \subset \cdots \underline\E'_{m'}=\underline{\hat \E}'$$
be the Harder-Narasimhan filtrations of $\underline {\hat\E}$ and $\underline{\hat \E}'$ respectively. If we denote $\mu_i:=\mu(\underline \E_i/\underline \E_{i-1})$ and 
$ \mu_i':=\mu(\underline \E'_{i} /\underline \E'_{i-1})$, then by assumption we have 
$$\mu_1-\mu_m<1, \mu_1'-\mu_m'<1,$$ 
and there exists a natural isomorphism  $
\rho: \hat{\E}|_{\hat{B}\setminus D}\rightarrow\hat{\E}'|_{\hat{B}\setminus D}.$  By Proposition \ref{prop3-7}, $\rho$ is a meromorphic section of $\hat{\E}^*\otimes \hat\E'$. Suppose $\det\rho$ has a pole of order $k\in \Z$ along $D$. If we write $k=d\cdot \rk(\E)+k_0$ with $0\leq k_0<\rk(\E)$, then by replacing $\hat\E$ with $\hat{\E}(d[D])$ and $\rho$ by $\rho\otimes s_D^{\otimes d}$ we may assume $0\leq k<\rk(\E)$. 

Denote 
$$\underline \rho=\iota^* \rho, \ \ \ \ \underline\rho^{-1}=\iota^* \rho^{-1}.$$
Then  $\underline \rho$ and $\underline \rho^{-1}$ can be viewed as two \emph{nontrivial} holomorphic sections 
$$\underline \rho: \underline{\hat{\E}}\rightarrow \underline{\hat{\E}}'(-l_0), \ \ \ \  \underline \rho^{-1}: \underline{\hat{\E}}' \rightarrow \underline{\hat{\E}}(-l'_0),$$
for some $l_0,l_0'\in \mathbb{Z}_{+}$. Let $k$ be the smallest integer such that $\underline\rho|_{\underline\E_{k+1}}\neq 0$. Then $\underline \rho$ descends to be a nontrivial holomorphic map $\underline \rho: \underline{\hat\E}/\underline \E_{k} \rightarrow \underline{\hat{\E}'}(-l_0)$ which restrict to be nonzero on $\underline \E_{k+1}/\underline \E_{k}$. Since $\underline \E'_1(-l_0)$ is the maximal destabilizing subsheaf of $\underline{\hat{\E}}'(-l_0)$, we have $\mu'_1-l_0 \geq \mu_{k+1}.$ Similarly $\mu_1-l'_0\geq\mu'_j$ for some $j$.  Then we have
$$2>\mu_1'-\mu_j'+\mu_1-\mu_{k+1} \geq l_0+l_0',$$
which implies exactly one of the following hold
\begin{itemize}
\item[(a).] $l_0=0$;
\item[(b).] $l_0=1$.
\end{itemize}
Suppose first (a) holds, then by assumption, $\rho$ can be extended to be a holomorphic section across $D$ and thus $\det(\rho)$ is also a holomorphic section of $\det(\hat{\E}^*)\otimes \det(\hat{\E}')$ over $\hat{B}$. However, by assumption we know $\det(\rho)$ has a pole of order $k_0\geq 0$. Then we must have $k_0=0$, i.e. $\det(\rho)|_{D}\neq 0$ which implies $\det(\rho)(z)\neq 0$ for any $z\in \hat{B} \setminus \text{Sing}(\hat{\E})\cup \text{Sing}(\hat{\E}')$. In particular, $\rho$ is an isomorphism away from complex codimension two and hence must be an isomorphism.  Notice this already finishes the proof of Part (II) of Theorem \ref{main} since under the assumption of (II) we know (a) must hold.

Now suppose (b) holds, i.e. $l_0=1$ and $l'_0=0$. By assumption, $\rho$ can be viewed as a holomorphic map $\rho: \hat{\E} \rightarrow \hat{\E}'([D])$
with $\underline\rho: \underline{\hat{\E}} \rightarrow \underline{\hat{\E}}'(-1)$ being nonzero and  $\rho^{-1}: \hat{\E}' \rightarrow \hat{\E}$ is a holomorphic map with $\underline\rho^{-1}:\underline{ \hat{\E}}' \rightarrow \underline{\hat{\E}}$ being nonzero. Then $\rho^{-1}$ is a sheaf monomorphism since $\hat{\E}'$ is reflexive and $\text{ker}(\rho^{-1})$ is supported on $D$. In the following, we do not distinguish between $\hat{\E}'$ and the image $\rho^{-1}(\hat{\E}')$ in $\hat{\E}$. Let $D'=Sing(\hat{\E})\cup Sing(\hat{\E}') \cup Sing(\underline{\hat{\E}}/\underline\E_{k})$.

\

To finish the proof of (III), it suffices to prove
\begin{clm}\label{exact}
$(\hat{\E}/\hat{\E}')|_{\hat{B}\setminus D'}\cong \iota_* (\underline {\hat{\E}}/\underline{\E}_{k})|_{\hat{B}\setminus D'}.$ 
\end{clm}
Indeed, given Claim \ref{exact}, we have the following exact sequence outside $D'$
$$
0\rightarrow \hat{\E}'\rightarrow\hat{\E} \rightarrow \iota_{*}(\underline {\hat{\E}}/\underline{\E}_{k})\rightarrow 0.
$$ 
By definition, we have $\hat{\E}'=\hat{\E}^k$ outside $D'$ where $\hat{\E}^k$ denotes the Hecke transform of $\hat{\E} $ along $\underline \E_k$. Since $\hat{\E}'$ and $\hat{\E}^k$ are both reflexive, they must be isomorphic. 

\begin{proof}[Proof of Claim \ref{exact}]

First we prove that $\hat{\E}/\hat{\E}'=\iota_* \iota^*(\hat{\E}/\hat{\E}')$.  To see this it suffices to show that for any local section $s$ of $\hat{\E}$, $z_n s \in \hat{\E}'$. Here $z_n$ denotes the local defining function for $D$ after choosing a local coordinate. Indeed, by assumption, $z_n \rho(s)$ is a local holomorphic section. We also know that $\rho^{-1}(z_n\rho(s))=z_n s$, which implies  $\I_D\hat{\E}  \subset \rho^{-1}(\hat{\E}')$. As a result, $\iota_* \iota^*(\hat{\E}/\hat{\E}')=\hat{\E}/\hat{\E}'$.

So it suffices to prove $\iota^*(\hat\E/\hat{\E}')=\underline {\hat\E}/\underline\E_k$ on $D\setminus D'$. Since all these sheaves are locally free away from $D'$ this boils down to showing  $\underline \rho^{-1}(\underline{\hat{\E}'})=\underline{\E}_{k}$  on $D\setminus D'$. 

We first show $\text{Im}(\underline\rho^{-1})\subset \underline \E_{k}$. If not, there exists a nontrivial map 
$$
\underline \rho^{-1}: \underline{\hat{\E}}' \rightarrow  \underline{\hat\E}/\underline\E_{k}
$$
which implies $\mu_j'\leq \mu_{k+1}$ for some $j$. Meanwhile, by assumption, $\underline\rho$ descends to be a nontrivial map as $\underline \rho: \underline{\hat{\E}}/\underline{\E}_{k} \rightarrow \underline{\hat{\E}}'(-1)$ which implies $\mu_1' -1\geq \mu_{k+1}$. Then we have 
$$
\mu_1'-\mu_{m'}' \geq \mu_1'-\mu_j'\geq 1
$$ 
which is a contradiction. Now we prove that $\text{Im}(\underline\rho^{-1}(z))= \underline \E_{k}|_z$ for $z\in D\setminus D'$.  It suffices to prove 
$$
\rk(\underline\rho(z))+\rk(\underline\rho^{-1}(z))\geq \rk(\E)
$$
for $z\in D\setminus D'$. Now we fix $z\in D\setminus D'$ and choose  local coordinates $ (z_1,\cdots z_n )$ so that $z_n$ is the local defining function for $ D$. After choosing a local trivialization for both $\hat{\E}$ and $\hat{\E}'$ near $z$, we can view $\rho$ and $\rho^{-1}$ as a matrix. By doing Taylor expansion, we can assume
$$
\rho^{-1}=A_0+A_1 z_n +\cdots
$$
and 
$$
z_n \rho=B_0+ B_1z_n+\cdots$$
where $A_i$ and $B_i$ are matrices of holomorphic functions independent of $z_n$. Since $\rho^{-1} \circ (z_n\rho) =z_n \Id$,  by comparing the coefficients in front of $z_n$ we get 
$$A_0 B_1+A_1B_0=\Id, $$
 which implies 
$$ 
\begin{aligned}
\rk(A_0)+\rk(B_0)
&\geq \rk(A_0B_1)+\rk(A_1B_0) \\
&\geq \rk(A_0B_1+A_1B_0)\\
&=\rk(\E).
\end{aligned}
$$
By definition, we have
$$\rk(\underline\rho(z))+\rk(\underline\rho^{-1}(z))=\rk(A_0)+\rk(B_0)\geq \rk(\E).$$
 This then finishes the proof. 
\end{proof}

\subsection{Proof of (IV)}\label{proofIV}
Now we assume $\E$ is homogeneous i.e. $\E\simeq \psi_*\pi^*\underline \E$ for reflexive $\underline \E$ over $\C\P^{n-1}$. Let $0=\underline \E_0 \subset \underline \E_1 \cdots \subset \underline\E_m=\underline \E$ be the Harder-Narasimhan filtration of $\underline \E$ and denote $\mu_k=\mu(\underline{\E}_k/\underline{\E}_{k-1})$. Note $\phi^*\underline \E \in \mathcal{A}$. Let $j_0=0$ and define
$$
j_{k+1}:=\max\{s> j_{k}: \mu_1-\mu_{s}-\lfloor{\mu_1-\mu_{j_k+1}}\rfloor<1, s\leq m\}
$$
inductively for $k\geq 1$. Let $l$ be the largest integer so that $j_l$ is defined. Then we define the \emph{partial} Harder-Narasimhan filtration as 
$$
0=\underline\E_{j_0} \subset \underline\E_{j_1} \subset \underline\E_{j_2} \subset \cdots \underline\E_{j_l} \subset \underline \E.
$$
Let $n_k=\lfloor{\mu_1-\mu_{j_k+1}}\rfloor$ for $0\leq k\leq l-1$ and define 
$$\widetilde{Gr}(\underline \E):=\oplus^{l}_{i=1} (\underline{\E}_{j_i}/\underline{\E}_{j_{i-1}})(n_{i-1}).$$
Then to prove (IV), it suffices to show
\begin{prop}
There exists an optimal extension $\hat{\E}\in \mathcal{A}$ so that $\underline{\hat{\E}}\cong\widetilde{Gr}(\underline \E)$. 
\end{prop}

\begin{proof}
It suffices to prove the following by induction on $k$ with $1\leq k \leq l-1$. (The reason to write inductions in this way will be justified by the proof naturally.)
\begin{itemize}
\item[$(a)_k$] there exists $\hat{\E}^k \in \mathcal{A}$ with $\underline{\hat{\E}}^k\cong\oplus_{i=1}^k (\underline\E_{j_i}/\underline\E_{j_{i-1}})(n_{i-1})\oplus (\underline \E/\underline \E_{j_k})(n_k)$;
\item[$(b)_k$] there exists the following sheaf inclusions for $1\leq i\leq k$ which are compatible with the splittings in $(a)_k$ 
\begin{itemize}
\item $\phi^*\underline \E_{j_1} \subset \hat{\E}^k$;
\item Let $\hat{\E}^k_1:=\hat{\E}^k$,  then we can define $\hat{\E}^k_{i+1}=\hat{\E}^k_i/\phi^*((\underline \E_{j_{i}}/\underline \E_{j_{i-1}})(n_{i-1}))$ for $1\leq i\leq k-1$ inductively and $\phi^*((\underline \E_{j_{i+1}}/\underline \E_{j_{i}})(n_{i}))\subset\hat{\E}^k_{i+1}$ for $i=1,\cdots k-1$;
\item $\hat{\E}^k_k/\phi^*((\underline \E_{j_{k}}/\underline \E_{j_{k}})(n_{k-1}))=\phi^*((\underline \E/\underline \E_{j_{k}})(n_{k}))$.
\end{itemize}
\end{itemize}
For $k=1$, we let $\hat{\E}^{1,1}$ be the Hecke transform of $\phi^*\underline \E$ along $\underline \E_{j_1}$. By Proposition \ref{Hecktransform}, we have the following exact sequence
$$
0\rightarrow (\underline{\E}/\underline{\E}_{j_1})(1)\rightarrow \underline{\hat{\E}}^{1,1} \rightarrow \underline \E_{j_1}\rightarrow 0.
$$
By definition, there exists a natural sheaf inclusion $ \phi^*\underline{\E}_{j_1}\subset\hat{\E}^{1,1}$ which restricts to be a map from $\underline \E_{j_1}$ to $\underline{\hat{\E}}^{1,1}$ that splits the exact sequence above i.e.  $\underline{\hat{\E}}^{1,1}\cong\underline \E_{j_1}\oplus (\underline{\E}/\underline{\E}_{j_1})(1)$. Indeed, we know that $\phi^*\underline \E_{j_1}$ lies in the kernel of the surjective map $\phi^*\underline\E \rightarrow \iota_* (\underline \E/\underline \E_{j_1})$ and thus we have a natural sheaf inclusion $\phi^*\underline \E_{j_1}\subset\hat{\E}^{1,1}$ by definition. (This is the key difference in the homogeneous case from the general case where we have a natural inclusion $\phi^*(\underline \E_{j_1}) \subset \hat{\E}^{1,2}$. ) The restriction map splitting the exact sequence above is tautological. Moreover, by definition, we have
$$
0\rightarrow \hat{\E}^{1,1}/\phi^*(\underline \E_{j_1}) \rightarrow \phi^*(\underline\E/\underline{\E}_{j_1}) \rightarrow \iota_*(\underline \E/\underline \E_{j_1}) \rightarrow 0
$$
which implies $\hat{\E}^{1,1}/\phi^*(\underline \E_{j_1})=\phi^*(\underline \E/\underline \E_{j_1})(-[D])=\phi^*(\underline \E/\underline \E_{j_1}(1))$. (This is another key difference in the homogeneous case from the general case. That is the quotient sheaf $\hat{\E}^{1,2}/\phi^*\underline \E_{j_1}$ is still homogeneous , i.e. it is pulled back from the projective space. ) If $n_1>1$, let $\hat{\E}^{1,2}$ be the Hecke transform of $\hat{\E}^{1,1}$ along $\underline \E_{j_1}$. Similarly, we have 
$$
0\rightarrow (\underline{\E}/\underline{\E}_{j_1})(2)\rightarrow \underline{\hat{\E}}^{1,2} \rightarrow \underline \E_{j_1}\rightarrow 0
$$
and by definition, we have a sheaf inclusion $\phi^*\underline{\E}_{j_1}\subset \hat{\E}^{1,2}$ which restricts to be a map that splits the exact sequence above i.e. $\underline{\hat{\E}}^{1,2}\cong (\underline{\E}/\underline{\E}_{j_1})(2)\oplus \underline \E_{j_1}$. By definition,  we also have the following exact sequence 
$$
0\rightarrow \hat{\E}^{1,2}/\phi^*\underline \E_{j_1} \rightarrow \phi^*((\underline \E/\underline \E_{j_1})(1) \rightarrow \iota_*((\underline \E/\underline \E_{j_1})(1)) \rightarrow 0
$$
which implies $\hat{\E}^{1,2}/\phi^*\underline \E_{j_1}=\phi^*((\underline \E/\underline \E_{j_1})(2))$. Then one can keep doing Hecke transform for $\hat{\E}^{1,2}$ along $\underline \E_{j_1}$ if necessary and get $\hat{\E}^1:=\hat{\E}^{1,n_1}\in\mathcal{A}$ satisfying
\begin{itemize}
\item[$(a)_1$] $\underline{\hat{\E}}^1\cong \underline\E_{j_1}\oplus (\underline \E/\underline \E_{j_1})(n_1)$;
\item[$(b)_1$] there exists a sheaf inclusion $\phi^*\underline \E_{j_1} \subset \hat{\E}^1$ which is compatible with the splitting above 
and $\hat{\E}^1/\phi^*(\underline \E_{j_1})=\phi^*(\underline \E/\underline \E_{j_{1}}(n_1))$.
\end{itemize}
Namely, after we do Hecke transform along $\underline \E_{j_1}$, $\phi^*\underline \E_{j_1}$ will always be a saturated subsheaf of the new sheaf which will give a splitting on the central fiber. And the natural quotient sheaf is still homogeneous. In the case of sub-bundles, one can use the bundle construction in Section \ref{Section2.1} to achieve the above result in one step.  

\

To make the argument more clear, we will explain how to do $k=2$ briefly. (Details can be found in the induction for the general case. )  Given $(a)_1$ and $(b)_1$, we can keep doing Hecke transform along $\phi^*\underline \E_{j_1} \oplus \phi^*(\underline \E_{j_2}/\underline \E_{j_1}(n_1))$ to get a new sheaf $\hat{\E}^2$. And we have two sheaf inclusions $\phi^*\underline \E_{j_1} \subset \hat{\E}^2$ and $\phi^*(\underline \E_{j_2}/\E_{j_1}(n_1))\subset \hat{\E}^2/\phi^*\underline \E_{j_1}$ which restricts to be maps that split the central fiber as we want. Furthermore, we have 
$$
(\hat{\E}^2/\phi^*(\underline \E_{j_1}))/\phi^{*}(\underline \E_{j_2}/\underline \E_{j_1} (n_1))=\phi^*(\underline \E /\underline \E_{j_2} (n_2))
$$
where $n_2$ is equal to the number of Hecke transforms along $\phi^*\underline \E_{j_1} \oplus \phi^*(\underline \E_{j_2}/\underline \E_{j_1}(n_1))$ to $\hat{\E}^2$. 

Now we do the induction in general. Suppose we have proved $(a)_k, (b)_k$, we want to build the statements $(a)_{k+1}$ and $(b)_{k+1}$. First let $\hat{\E}^{k+1,1}$ to be the Heck transform of $\hat{\E}^k$ along $\oplus_{i=1}^k (\underline\E_{j_i}/\underline\E_{j_{i-1}})(n_{i-1})\oplus (\underline \E_{j_{k+1}}/\underline \E_{j_k})(n_k)$.  By Proposition \ref{Hecktransform} we have the following exact sequence
$$
0\rightarrow (\underline \E/\underline \E_{j_{k+1}})(n_k+1)\rightarrow 
\underline{\hat{\E}}^{k+1,1}\rightarrow \oplus_{i=1}^k (\underline\E_{j_i}/\underline\E_{j_{i-1}})(n_{i-1})\oplus (\underline \E_{j_{k+1}}/\underline \E_{j_k})(n_k)\rightarrow 0.
$$
Then $(b)_{k}$ holds by replacing $\hat{\E}^k$ with $\hat{\E}^{k+1,1}$ except the last one which needs to be changed. More precisely, there exists the following sheaf inclusions for $1\leq i\leq k$ which are compatible with the splittings in $(a)_k$ 
\begin{itemize}
\item $\phi^*\underline \E_{j_1} \subset \hat{\E}^{k+1,1}$;
\item if we let $\hat{\E}^{k+1}_1:=\hat{\E}^{k+1,1}$ and define $\hat{\E}^{k+1}_{i+1}=\hat{\E}^{k+1}_i/\phi^*((\underline \E_{j_{i}}/\underline \E_{j_{i-1}})(n_{i-1}))$ for $1\leq i\leq k-1$ inductively, then $\phi^*((\underline \E_{j_{i+1}}/\underline \E_{j_{i}})(n_{i}))\subset\hat{\E}^{k+1}_{i+1}$ for $i=1,\cdots k-1$;
\item $\phi^*((\underline \E_{j_{k+1}} /\underline \E_{j_k})(n_k))\subset \hat{\E}^{k+1,1}_{k+1}$ and 
$$\hat{\E}^{k+1,1}_{k+1}/\phi^*((\underline \E_{j_k+1} /\underline \E_{j_k})(n_k))=\phi^*(\underline\E/\underline\E_{j_{k+1}} (n_k+1)).$$
\end{itemize}
Indeed, by definition, we have 
$$
0\rightarrow \hat{\E}^{k+1,1}\rightarrow \hat{\E}^k \rightarrow \iota_*(\oplus_{i=1}^k (\underline\E_{j_i}/\underline\E_{j_{i-1}})(n_{i-1})\oplus (\underline \E_{j_{k+1}}/\underline \E_{j_k})(n_k))\rightarrow 0
$$
Combining this with that $\hat{\E}^k$ satisfies property $(a)_k$ and $(b)_k$, we can easily get the sheaf inclusions with required properties above.  Now we have  
$$
\underline{\hat{\E}}^{k+1,1}=\oplus_{i=1}^{k+1} (\underline\E_{j_i}/\underline\E_{j_{i-1}})(n_{i-1})\oplus (\underline \E /\underline \E_{j_{k+1}})(n_k+1).
$$ 
Indeed, the sheaf inclusion $\phi^*\underline \E_{j_1} \subset \hat{\E}^{k+1,1}$ restricts to be a map that gives a splitting  $\hat{\E}^{k+1,1}=\underline \E_{j_1} \oplus \iota^*\hat{\E}^{k+1}_2$. For $\iota^*\hat{\E}^{k+1}_2$, the sheaf inclusion given by $\phi^*((\underline \E_{j_{2}}/\underline \E_{j_{1}})(n_{1}))\subset\hat{\E}^{k+1}_{2}$ gives a splitting   
$ \iota^*\hat{\E}^{k+1}_2=(\underline \E_{j_{2}}/\underline \E_{j_{1}})(n_{1})\oplus \iota^*\hat{\E}^{k+1}_3$. Then one can keep doing this and finally get a splitting of $\underline{\hat \E}^{k+1,1}$ as claimed above.

Now one can repeat the process with $\hat{\E}^{k+1,1}$ to get $\hat{\E}^{k+1,2}$ by doing Hecke transform along $\oplus_{i=1}^k (\underline\E_{j_i}/\underline\E_{j_{i-1}})(n_{i-1})\oplus (\underline \E_{j_{k+1}}/\underline \E_{j_k})(n_k)$ again if necessary and finally get $\hat{\E}^{k+1}:=\hat{\E}^{k+1,n_{k+1}}$ satisfying properties $(a)_{k+1}$ and $(b)_{k+1}$. This finishes the proof. 
\end{proof}
\begin{rmk}\label{splitting}
When the Harder-Narasimhan filtration of $\underline{\E}$ has length equal to $2$, i.e. 
$$
0=\underline\E_0 \subset \underline \E_1 \subset \underline\E_2=\underline \E,
$$
 the same argument shows that there exists an optimal extension $\hat{\E}$ so that $\underline{\hat \E}=\underline \E_1 \oplus (\underline\E_2/\underline \E_1)(k)$ for some integer $k$ with $\mu_1-1<\mu_2-k\leq \mu_1$. In general, one should not expect to get an optimal extension of which the restriction splits as a direct sum of semistable torsion free sheaves by Theorem \ref{main} (III) and Corollary \ref{lemma1.14}.
\end{rmk}
 
\section{Examples}\label{Examples}

In this section, we apply Theorem \ref{main} to study some interesting examples.
 
\

\noindent\textbf{Example 1.} Consider $\underline \E \rightarrow \C\P^2$ given by the following exact sequence 
$$
0\rightarrow \O\xrightarrow{\sigma} \O(1)\oplus \O(1)\oplus \O(3) \rightarrow \underline\E
\rightarrow 0,$$
where $\sigma=(z_1,z_2, z_3^k)$. Consider $\E=\psi_*\pi^*\underline \E$. Then we have (see Section $5$ in \cite{CS2})
\begin{itemize}
\item if $k=1$, $\underline \E$ is stable;
\item if $k=2$, $\underline \E$ is semistable;
\item if $k\geq 3$, $\underline \E$ is unstable. The Harder-Narasimhan filtration of $\underline \E$ (which is the same as the Harder-Narasimhan-Seshadri filtration in this case) is given by
$0\subset\underline \E_1 \subset \underline \E_2=\underline \E$ where $\underline \E_1\cong \O(k)$ and $\underline \E_2 /\underline \E_1 \cong \mathcal{I}_{[0:0:1]}(2)$. 
\end{itemize}
By Theorem \ref{main}, when $k\leq 2$, there exists a unique optimal extension given by $\phi^*\underline \E$ (up to equivalence). When $k\geq 3$,  
by Remark \ref{splitting}, there exists an optimal extension $\hat{\E}$ of which the restriction is given by $\O(2)\oplus \mathcal{I}_{[0,0,1]}(2)$. Then again by Theorem \ref{main}, $\hat{\E}$ is the unique one up to equivalence since $\O(2)\oplus \mathcal{I}_{[0,0,1]}(2)$ is semistable.  These are compatible with our study of analytic tangent cones in \cite{CS1, CS2}. 

\

The next is an example where there are two optimal extensions, for which one of them has a locally free algebraic tangent cone while the other has an essential point singularity. 

\

\noindent\textbf{Example 2.} Consider a vector bundle $\underline\E \rightarrow \C\P^{3}$ given by the following
\begin{equation}\label{eqn2.1}
0\rightarrow \O\xrightarrow{\sigma}  \O(1)^{\oplus 3}\oplus \O(2)\rightarrow \underline\E \rightarrow 0.
\end{equation}
where $\sigma=(z_1, z_2, z_3, z_4^2)$. Let $\E:=\psi_*\pi^*\underline \E$. Then $\hat{\E}:=\phi^*\underline \E$ is an optimal extension of $\E$ at $0$ with $\Phi(\hat{\E})=\frac{1}{2}$. The Harder-Narasimhan filtration of $\underline{\hat\E}$ is given by $ \underline\E_1\cong \O(2)$ and $\underline \E_2=\underline\E$. Furthermore, $\underline{\E}_2/\underline{\E}_1$ fits into the following exact sequence 
$$
0\rightarrow \O\xrightarrow{\sigma'} \O(1)^{\oplus 3} \rightarrow \underline \E_2 /\underline \E_1\rightarrow 0
$$
where $\sigma'=(z_1, z_2, z_3)$. In particular, $\underline \E_2/\underline \E_1$ is a stable reflexive sheaf with an essential point singularity
at $[0,0,0,1]$. Let $\hat{\E}^1$ be the Hecke transform of $\hat{\E}$ along $\underline\E_1$ which is again an optimal extension. By Remark \ref{splitting}, $\underline{\hat \E}^1= \underline{\E}_1\oplus (\underline{\E}_2/\underline \E_1)(1)$. In particular, $\hat{\E}^1$ is an optimal extension of which the restriction splits as a direct sum of stables sheaves which has an essential point singularity. 

\


\begin{thebibliography}{1}
\bibitem{BS} S. Bando, Y-T. Siu. \emph{Stable sheaves and Einstein-Hermitian metrics},  Geometry and Analysis on Complex Manifolds, T. Mabuchi, J. Noguchi, and T. Ochiai, eds., World Scientific, River Edge, NJ, 1994, 39--50.
\bibitem{CS1} X-M. Chen, S. Sun. \emph{Singularities of Hermitian-Yang-Mills connections and the Harder-Narasimhan-Seshadri filtration}  arXiv:1707.08314 (2017).
\bibitem{CS2} X-M. Chen, S. Sun. \emph{Analytic tangent cones of admissible Hermitian-Yang-Mills connections}  arXiv:1806.11247 (2018).
\bibitem{Demailly}J-P. Demailly. \emph{Analytic methods in algebraic geometry.} Somerville, Mass: International Press, 2012.
\bibitem{Gasparim} E. Gasparim. \emph{Rank two bundles on the blow up of $\C^2$}, J. Algebra 199 no. 2, 581---590 (1998) 
\bibitem{Shiffman}B. Shiffman. \emph{On the removal of singularities of analytic sets.} Michigan Math. J 15.1 (1968): 1.
\bibitem{Siu}Y-T. Siu.  \emph{A Hartogs type extension theorem for coherent analytic sheaves.} Annals of Mathematics 166-188 (1971).
\bibitem{Tian} G. Tian. \emph{Gauge theory and calibrated geometry. I}.  Ann. of Math. (2) 151.1 (2000), 193--268.


\end{thebibliography}
\end{document}